\documentclass[reqno,A4paper]{amsart}
\usepackage{amsmath,amssymb,amscd,amsthm,verbatim,alltt,amsfonts,array}
\usepackage{mathrsfs}
\usepackage[english]{babel}
\usepackage{latexsym}
\usepackage{amssymb}
\usepackage{euscript}
\usepackage{graphicx}
\usepackage{tikz}
\usepackage{xcolor}
\usepackage{csquotes}
\graphicspath{ {./images/} }
\newtheorem{thm}{Theorem}[section]
\newtheorem{cor}[thm]{Corollary}
\newtheorem{prop}[thm]{Proposition}

\newtheorem{lem}[thm]{Lemma}

\newtheorem{ex}{Example}[section]
\newcommand{\be}{\begin{equation}}
\newcommand{\ee}{\end{equation}}
\newcommand{\ben}{\begin{enumerate}}
\newcommand{\een}{\end{enumerate}}
\newcommand{\beq}{\begin{eqnarray}}
\newcommand{\eeq}{\end{eqnarray}}
\newcommand{\beqn}{\begin{eqnarray*}}
\newcommand{\eeqn}{\end{eqnarray*}}

\title{New Classes of Finsler Metrics:\\ The birth of new Projective Invariant}
\author{Nasrin Sadeghzadeh}
\newcommand{\acr}{\newline\indent}
\address{\llap{*\,}Department of Mathematics,\acr
University of Qom, \acr
Alghadir Bld\acr
Qom\acr
Iran}
\email{nsadeghzadeh@qom.ac.ir}
\begin{document}
\maketitle
\section{Abstract}
This paper presents a pioneering projective invariant in Finsler geometry, introducing a new class of Finsler metrics that are preserved under projective transformations. The newly formulated weakly generalized Douglas-Weyl $(W-G D W)$ equation facilitates the generalization of generalized Douglas-Weyl $(G D W)$-metrics into the broader category of $W-G D W$-metrics, which encompasses all $G D W$-metrics. Within this class, there are also two additional subclasses: generalized weakly-Weyl metrics, characterized by a milder form of Weyl curvature, and generalized $\tilde{D}$-metrics, defined by a less strict version of Douglas curvature. The paper provides a comprehensive overview of these generalized class of Finsler metrics and elucidates their properties, supported by detailed examples.
\vspace{4mm}

\textbf{2010 Mathematics Subject Classification}: { 53B40; 53C60}

\textbf{Keywords}: {Weakly generalized Douglas-Weyl metrics, Generalized Douglas-Weyl metrics, Generalized $\tilde{D}$-metrics, Generalized weakly-Weyl metrics.}
\section{Introduction}
In the intricate fabric of Finsler geometry, projective invariants act as essential threads that illuminate the deep connections between a manifold's properties and the intrinsic characteristics of its geodesics. Finsler geometry is a metric generalization of Riemannian geometry, where the general definition of the length of a vector is not necessarily given in the form of the square root of a quadratic form as in the Riemannian case. In fact, Finsler geometry is better described as Riemannian geometry without the quadratic restriction \cite{Shiing}.

The projective change between two Finsler spaces is an important concept in Finsler geometry, as it helps to understand the relationship between different Finsler metrics on the same manifold. These transformations preserve specific projective properties of the underlying geometry, revealing essential aspects of the geometry and curvatures of Finsler spaces. This concept has been studied extensively by many researchers \cite{Proj. 1}, \cite{Proj. 3}, \cite{Proj. 4} and \cite{Proj. 5}.

In the realm of projective geometry, projective invariants hold immense significance and are exemplified by well-known cases like Weyl, Douglas and the Generalized Douglas-Weyl (GDW)metrics. Douglas, Weyl and $G D W$-metrics are the fundamental quantities in projective Finsler geometry. Douglas metrics are characterized by vanishing Douglas curvature \cite{Douglas}, \cite{Bacso}, while Weyl metrics are characterized by vanishing Weyl curvature \cite{Sh2}. A Finsler metric is a Weyl metric if and only if it is of scalar flag curvature [16]. The study of projective invariants in Finsler geometry has led to the emergence of the new classes of Finsler metrics, such as the Generalized Douglas-Weyl Finsler metrics $(G D W)$-metrics \cite{GDW}. The class of Generalized Douglas-Weyl Finsler metrics $(G D W(M))$ are a class of Finsler metrics that satisfy the equation $D_{j}{^i}_{k l \mid m} y^{m}=T_{j k l} y^{i}$ for some tensor's coefficients denoted by $T_{j k l}$, where $D_{j}{^i}{ }_{k l \mid m}$ denotes the horizontal covariant derivatives of $D_{j}{^i}_{k l}$ with respect to the Berwald connection of $F$.

Projective invariants capture the essential geometric properties that remain unchanged under projective changes on a Finsler manifold $(M, F)$.

This paper aims to present a significant breakthrough in Finsler geometry by unveiling a new projective invariant. The development of projective invariants in Finsler geometry, particularly the introduction of the Weakly Generalized Douglas-Weyl $(W-G D W)$-metrics, represents a crucial step in advancing our understanding of projective Finsler manifolds and their geometric properties. The $W-G D W$ metrics, a new class of Finsler metrics, are closed under projective changes, signifying a notable expansion in the study of Finsler geometry. The class of $W-G D W$-metrics comprises a category of Finsler metrics that fulfill the following equation
$$
\tilde{D}_{j}{ }^{i}{ }_{k l \mid m} y^{m}+\lambda F \tilde{D}_{j}{ }^{i}{ }_{k l}=U_{j k l} y^{i} \text {. }
$$
Here $U_{j k l}$ are the cofficients of a suitable tensor field, and $\lambda=\lambda(x, y)$ is a scalar function defined on the tangent bundle $T M$. Moreover, $\tilde{D}_{j}{ }^{i}{ }_{k l}=D_{j}{ }^{i}{ }_{k l \mid m} y^{m}$ and $D_{j}{ }^{i}{ }_{k l \mid m}$ represrnts the horizontal covariant derivatives of $D_{j}{ }^{i}{ }_{k l}$ with respect to the Berwald connection of $F$.

This new class of Finsler metrics includes all previous class of projective invariant Finsler metrics, such as Douglas, Weyl, and GDW-metrics.\\
Additionally, it contains two new subclasses of Finsler metrics: generalized weakly-Weyl and generalized $\tilde{D}$-metrics.\\
The weakly-Weyl subclass is a projective invariant class of Finsler metrics characterized by a weaker form of Weyl invariance. This subclass includes Weyl and W-quadratic Finsler metrics \cite{Dtilde-Sadegh}.\\
The generalized $\tilde{D}$-metrics constitute a class of Finsler metrics that contains Douglas metrics and is characterized by a weaker form of Douglas curvature. The Figure \ref{Figure} clarifies the hierarchical structure of these new classes of Finsler metrics and highlights their interrelations. It provides a comprehensive overview that enhances our understanding of how these new classes relate to the traditional invariant classes of Finsler metrics. The insights provided by Proposition \ref{GDtildinWGDW}, Corollary \ref{DinGDtild}, Corollary \ref{WeylnotGDtild}, Corollary \ref{WeylinGWWeyl} and Theorem \ref{GWWeylinWGDW} contribute significantly to the development of this hierarchical structure depicted in Figure \eqref{Figure}. \\
By introducing the new invariants, we aim to provide a comprehensive analysis of their properties. Additionally, we will investigate their connections to other known projective invariants and explore its application in characterizing Finsler manifolds.\\
Throughout this paper, the notations \enquote{${}_{.}$} and \enquote{${}_{\mid}$} represent the vertical and horizontal derivatives associated with the Berwald connection of Finsler metric $F$, respectively.\\
Additionally, the subscript \enquote{${}_0$} denotes the contraction by $y^{m}$ indicated by the subscript $m$
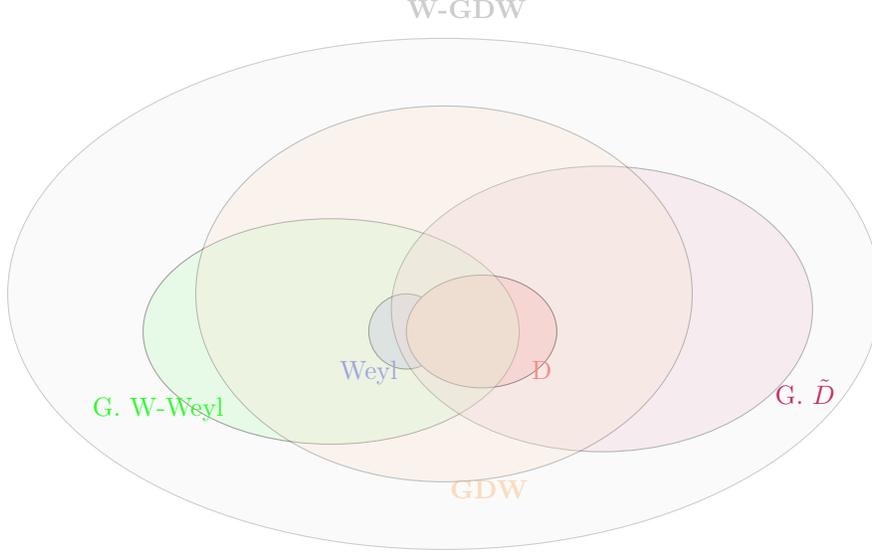
\begin{figure}
\begin{tikzpicture}
   \draw[fill=blue!20, opacity=8] (0,0) circle (0.5); \node[below=-0.04cm] at (-0.5, -0.3) {\textcolor{blue}{Weyl}};
  \draw[fill=red!20, opacity=2] (1,0) ellipse (1 and 0.75);  \node[below=0.02cm] at (1.8, -0.25){\textcolor{red}D};
  \draw[fill=green!20, opacity=0.45] (-1,0) ellipse (2.5 and 1.5); \node[below=0.05cm] at (-3.3, -0.7){\textcolor{green}{G. W-Weyl}};
  \draw[fill=purple!20, opacity=0.35] (2.6,0.3) ellipse (2.8 and 1.9); \node[below=0.01cm] at (5.3, -0.5){\textcolor{purple}{G. $\tilde{D}$}};
  \draw[fill=orange!20, opacity=0.3] (0.5,0.5) ellipse (3.3 and 2.5) node [below=1 ex] at (1.1 , -1.7){\textbf{\textcolor{orange}{GDW}}};
  \draw[fill=gray!20, opacity=0.2] (0.5,0.5) ellipse (5.8 and 3.4) node[above=1 ex]at (0.8 , 3.9) {\textbf{\textcolor{black}{W-GDW}}};
\end{tikzpicture}
\caption{\small The diagram illustrates the relationships between different classes of Finsler metrics. The blue circle represents the class of Weyl metrics, denoted as Weyl. The red circle represents the class of Douglas metrics, denoted as D. The green ellipse encompasses the class of generalized weakly-Weyl metrics, denoted as G. W-Weyl. The purple ellipse represents the class of generalized $\tilde{D}$-metrics, denoted as G. $\tilde{D}$. The orange ellipse denotes the class of generalized Douglas-Weyl metrics, denoted as GDW. The gray ellipse represents the class of weakly generalized Douglas-Weyl metrics, denoted as W-GDW.}
    \label{Figure}
\end{figure}
\section{Preliminaries}
A Finsler metric is defined on a manifold $M$ as a non-negative function $F$ on $TM$ that satisfies the following properties.
\ben
\item[(a)] $F$ is
 $C^{\infty}$ on $TM\setminus \{0\}$,
\item[(b)]
$F(\lambda y) =\lambda F(y)$, $\forall \lambda >0$, $\ y\in TM$,
\item[(c)] For each $y\in T_xM$,
the following quadratic form ${\bf g}_y$ on $T_xM$ is positive definite,
\be
    {\bf g}_y(u, v)= {1\over 2} \left[ F^2(y + su + tv) \right]_{s, t=0}, \quad u, v \in T_xM.
\ee
\een
At every point $x\in M$, a Finsler metric $F$ satisfies the property that $F_x= F|_{T_xM}$ is an Euclidean norm if and only if ${\bf g}_y$ is independent of $y\in T_xM\setminus{0}$.
A curve $c(t)$ is called a {\it geodesic} if it satisfies
\be
{d^2 c^i\over dt^2} + 2 G^i (c(t), \dot{c}(t))=0,
\ee
where $G^i(x, y)$ are local functions on $TM$ given by
\be\label{Gi}
G^i(x, y)= \frac{1}{4} g^{il}(x, y) \{ \frac{\partial^2 F^2}{\partial x^k \partial y^l} y^k - \frac{\partial F^2}{\partial x^l}\},\quad y\in T_xM,
\ee
and called the spray coefficients of $F=F(x,y)$. Here,
\[
G=y^i\frac{\partial }{\partial x^i}-2G^i(x,y) \frac{\partial}{\partial y^i},
\]
denotes the associated spray to $(M,F)$. The projection of an integral curve of $G$ is called a geodesic in $M$.\\
The Riemann curvature $R_y = R{^i}_k \frac{\partial}{\partial x^i}\otimes dx^k$ of $F$ is given by
\[
R{^i}_k=2\frac{\partial G^i}{\partial x^k}-\frac{\partial^2 G^i}{\partial x^m \partial y^k}y^m+ 2G^m\frac{\partial^2 G^i}{\partial y^m \partial y^k}-\frac{\partial G^i}{\partial y^m}\frac{\partial G^m}{\partial y^k}.
\]
For the Riemann curvature of Finsler metric $F$ one has \cite{Sh2}
\be\label{Rikl}
R{^i}_{kl}=\frac{1}{3}(R{^i}_{k.l}-R{^i}_{l.k}), \quad and \quad R_j{^i}_{kl}=R{^i}_{kl.j}.
\ee
$F$ is called a {\it Berwald metric} if $G^i(x, y)$ are quadratic in $y\in T_xM$ for all $x\in M$.
Define
\[
B_y:T_xM\times T_xM\times T_xM\rightarrow T_xM
\]
\[
B_y(u,v,w)=B_j{^i}_{kl}u^j v^k w^l \frac{\partial}{\partial x^i},
\]
where
$
B_j{^i}_{kl}=\frac{\partial^3 G^i}{\partial y^j \partial y^k \partial y^l}$ and $u=u^i\frac{\partial}{\partial x^i}$, $v=v^i \frac{\partial}{\partial x^i}$, $w=w^i\frac{\partial}{\partial x^i}$.
The relationship between Riemann and Berwald curvature is of significant interest, as noted in \cite{Sh2}.
\be\label{RieBer}
B_j{^i}_{ml|k}-B_j{^i}_{mk|l}=R_j{^i}_{kl.m}.
\ee
Define
\[
E_y:T_xM\times T_xM \rightarrow \mathbb{R},
\]
\[
E_y(u,v)=E_{jk} u^j v^k,
\]
where $E_{jk}=\frac{1}{2}B_j{^m}_{km}$. The Berwald curvature and mean Berwald curvature are denoted by $B$ and $E$, respectively. A Finsler metric $F$ is called a Berwald and Weakly Berwald (WB) metric if $B=0$ and $E=0$, respectively \cite{Sh3}. \\
A Finsler metric $(M,F)$ has isotropic mean Berwald curvature if
\[
E_{ij}=\frac{n+1}{2}cF^{-1} h_{ij},
\]
for some scalar function $c=c(x)$ on $M$, where $h_{ij}$ is the angular metric. The $S$-curvature $S(x,y)$ is defined as follows \cite{Sh3}
\[
S(x,y)=\frac{d}{dt}[\tau\Big(\gamma(t),\gamma'(t)\Big)]_{|t=0},
\]
where $\tau(x, y)$ is the distortion of the metric $F$ and $\gamma(t)$ is the geodesic with $\gamma(0)=x$ and $\gamma'(0)=y$ on $M$. It is known that \cite{Sh2}
\be\label{ES}
E_{ij}=\frac{1}{2}S_{.i.j}.
\ee
The Finslerian quantity $H$ was introduced by H. Akbar-Zadeh to characterization of Finsler metrics of constant flag curvature which is obtained
from the mean Berwald curvature $E$ by the covariant horizontal differentiation along geodesics. For a vector $y \in T_pM$,
\[
H_y : T_pM \times T_pM \longrightarrow \mathbb{R}
\]
is given by
\[
H_y(u, v) = H_{jk}(y)u^j v^k,
\]
where $H_{jk} = E_{jk|l}y^l$. 
Define
\be\label{Douglas}
D_j{^i}_{kl}=B_j{^i}_{kl}-\frac{1}{n+1}\frac{\partial^3}{\partial y^j \partial y^k \partial y^l}(\frac{\partial G^m}{\partial y^m}y^i).
\ee
The tensor $D=D_j{^i}_{kl} dx^j\otimes \frac{\partial}{\partial x^i}\otimes dx^k \otimes dx^l$ is a well-defined tensor on the slit tangent bundle $TM_0$, and is called the Douglas tensor. The Douglas tensor $D$ is a non-Riemannian projective invariant, meaning that if two Finsler metrics $F$ and $\bar{F}$ are projectively equivalent, i.e., if $G^i=\bar{G^i}+P y^i$ where the projective factor $P=P(x,y)$ is positively $y$-homogeneous of degree one, then the Douglas tensor of $F$ is the same as that of $\bar{F}$ \cite{DShen}, \cite{Sh2}.
One could easily show that
\be\label{D2}
D_j{^i}_{kl}=B_j{^i}_{kl}-\frac{2}{n+1}\{E_{jk}\delta^i_l+E_{jl}\delta^i_k+E_{kl}\delta^i_j+E_{jk.l}y^i\}.
\ee
The Douglas curvature, denoted by $D_j{^i}_{kl}$, is a projective invariant that is constructed from the Berwald curvature. Finsler metrics with $D_j{^i}_{kl}=0$ are called Douglas metrics. Additionally, metrics satisfying the following condition are called $GDW$-metrics, which are also projective invariants.
\[
{D_j}^i{}_{kl|m}y^m=T_{jkl}y^i,
\]
for some tensor's coefficients denoted by $T_{jkl}$, where ${D_j}^i{}_{kl|m}$ denotes the horizontal derivatives of ${D_j}{^i}_{kl}$ with respect to the Berwald connection of $F$.\\
Z. Shen proposed a non-Riemannian quantity $\tilde{B}$, derived from the Berwald curvature $B$, through covariant horizontal differentiation along Finslerian geodesics \cite{Sh2}. Extending the concept further, we define a metric based on the expanded notion of Douglas curvature, termed $\tilde{D}$-metric. Given a vector $y \in T_pM$, define
\[
\tilde{D}_y : T_pM \times T_pM \times T_pM \longrightarrow T_pM
\]
\[
\tilde{D}_y(u, v, w)= \tilde{D}_j{^i}_{kl} u^j v^k w^l \frac{\partial}{\partial x^i},
\]
where $\tilde{D}_j{^i}_{kl}={D_j}{^i}_{kl|0}=D_j{^i}_{kl|m}y^m$. For a vector $y \in T_pM$, we define \cite{Dtilde-Sadegh}
\[
\mathfrak{D}_y : T_pM \times T_pM \times T_pM \times T_pM\longrightarrow T_pM
\]
\[
\mathfrak{D}_y(u, v, w, z)= \mathfrak{D}_j{^i}_{klm} u^j v^k w^l z^m \frac{\partial}{\partial x^i},
\]
where $\mathfrak{D}_j{^i}_{klm}=2(\tilde{D}_j{^i}_{kl|m}-\tilde{D}_j{^i}_{km|l})$. A Finsler metric $(M, F)$ is called Stretch Douglas if $\mathfrak{D}_j{^i}_{klm}=0$. Additionally, if the metric satisfies the extra requirement below, it becomes an isotropic stretch Douglas metric
\[
\mathfrak{D}_j{^i}_{klm}=\lambda (D_j{^i}_{kl|m}-D_j{^i}_{km|l}),
\]
where $\lambda=\lambda(x,y)$ is scalar function on $TM$.
\begin{lem} \cite{Sh2}
Let $F$ and $\bar{F}$ be two projectively equivalent Finsler metrics on $M$. The Riemann curvatures are related by
\be\label{Rieproj}
\bar{R}{^i}_k=R{^i}_k+E\delta{^i}_k+\tau_k y^i,
\ee
where
\[
E=P^2-P_{|m}y^m, \quad \quad \tau_k=3(P_{|k}-PP_{.k})+E_{.k}.
\]
Here $P_{|k}$ denotes the covariant derivative of projective factor $P$ with respect to $\bar{F}$.\\
\end{lem}
Now, consider a Riemannian metric denoted by $\alpha$ and a 1 -form represented by $\beta$ on a manifold $M$. Additionally, let there be a smooth function $\varphi=\varphi(s)$ defined on the interval $\left(-b_{0}, b_{0}\right)$, where $b_{0}$ is given by $b_{0}=\sup _{x \in M}\|\beta\|_{x}$. Using these elements, we can define a function on the tangent bundle $T M$ as follows,
\[
F=\alpha \varphi(s), \quad \text { where }\quad s=\frac{\beta}{\alpha}.
\]
Provided that the function $\varphi$ and the value $b_{0}$ satisfy certain conditions, denoted as \eqref{12} and \eqref{13}, we can conclude that $F$ constitutes a Finsler metric on the manifold $M$. Metrics of this form are referred to as $(\alpha, \beta)$-metrics. It is noteworthy that Randers metrics represent a special subclass of these $(\alpha, \beta)$-metrics. Now, turning our attention to the specifics of $(\alpha, \beta)$-metrics, let us define $\alpha$ as $\alpha(x,y)=\sqrt{a_{i j}(x) y^{i} y^{j}}$, which serves as the Riemannian metric, while $\beta$ is expressed as $\beta(x,y)=b_{i}(x) y^{i}$, denoting the 1 -form on the manifold $M$. Let
\[
b= \|\beta\|_{x}=\sqrt{a^{i j}(x) b_{i}(x) b_{j}(x)}.
\]
By a direct computation, we obtain

$$
g_{i j}=\rho a_{i j}+\rho_{0} b_{i} b_{j}-\rho_{1}\left(b_{i} \alpha_{. j}+b_{j} \alpha_{. i}\right)+s \rho_{1} \alpha_{. i} \alpha_{. j},
$$

where $\alpha_{. i}=\frac{a_{i j} y^{j}}{\alpha}$, and
\begin{equation}
\rho=\varphi\left(\varphi-s \varphi^{\prime}\right), \quad \rho_{0}=\varphi \varphi^{\prime \prime}+\varphi^{\prime}\varphi^{\prime}, \quad \rho_{1}=s\left(\varphi \varphi^{\prime \prime}+\varphi^{\prime} \varphi^{\prime}\right)-\varphi \varphi^{\prime}. \label{11}
\end{equation}

By further computation, one obtains
\[
\operatorname{det}\left(g_{i j}\right)=\varphi^{n+1}\left(\varphi-s \varphi^{\prime}\right)^{n-2}\left[\left(\varphi-s \varphi^{\prime}\right)+\left(b^{2}-s^{2}\right) \varphi^{\prime \prime}\right] \operatorname{det}\left(a_{i j}\right).
\]
Using the continuity, one can easily show that
\begin{lem}\cite{Sh2}
$F=\alpha \varphi(s)$, with $s=\frac{\beta}{\alpha}$, is a Finsler metric on $M$ for any pair $(\alpha, \beta)$ with $\|\beta\|_{x}=b <b_{0}$, with $b_{0}>0$ if and only if $\varphi=\varphi(s)$ satisfies the following conditions
\be\label{12}
\varphi(s)>0,\quad \left(|s| \leq b_{0}\right),
\ee
\be\label{13}
\varphi(s)-s \varphi^{\prime}(s)+\left(b^{2}-s^{2}\right) \varphi^{\prime \prime}>0, \quad\left(|s| \leq b \leq b_{0}\right).
\ee
\end{lem}
Let
\[
r_{i j}=\frac{1}{2}\left(b_{i\| j}+b_{j \| i}\right), \quad s_{i j}=\frac{1}{2}\left(b_{i \| j}-b {j \| i}\right),
\]
where  \enquote{$\|$} denotes the horizontal derivative with respect to $\alpha$. Additionally, we assume that
\[
r_{j}=b^{i} r_{i j}, \quad s_{j}=b^{i} s_{i j}, \quad r_{i 0}=r_{i j} y^{j}, \quad s_{i 0}=s_{i j} y^{j}, \quad r_{0}=r_{j} y^{j}, \quad s_{0}=s_{j} y^{j}.
\]
Suppose that $G^{i}$ and $\bar{G}^{i}$ denote the coefficients of $F$ and $\alpha$, respectively, in the same coordinate system. By definition, we obtain the following identity
\begin{equation}\label{GeoCoealphaBeta}
G^{i}=\bar{G}^{i}+P y^{i}+Q^{i},
\end{equation}
where
\begin{align}
& P=\alpha^{-1} \Theta\left[r_{00}-2 \alpha Q s_{0}\right], \label{GeoP}\\
& Q^{i}=\alpha Q s_{0}^{i}+\Psi\left[r_{00}-2 \alpha Q s_{0}\right] b^{i},\label{GeoQi} \\
& Q=\frac{\varphi^{\prime}}{\varphi-s \varphi^{\prime}},\label{GeoQ}\\
& \Theta=\frac{\varphi \varphi^{\prime}-s\left(\varphi \varphi^{\prime \prime}+\varphi^{\prime} \varphi^{\prime}\right)}{2 \varphi\left(\left(\varphi-s \varphi^{\prime}\right)+\left(b^{2}-s^{2}\right) \varphi^{\prime \prime}\right)},\label{GeoTheta} \\
& \Psi=\frac{1}{2} \frac{\varphi^{\prime \prime}}{\left(\varphi-s \varphi^{\prime}\right)+\left(b^{2}-s^{2}\right) \varphi^{\prime \prime}}\label{GeoPsi}.
\end{align}

\section{The birth of new projective tensor}
The birth of a new projective invariant in Finsler spaces marks a significant advancement in the field of Finsler geometry. This section presents an innovative new class of Finsler metrics that possess a unique property of being closed under projective changes. A new projective invariant equation has facilitated this development. The $W-GDW$ equation enables us to generalize $GDW$-metrics as the $W-GDW$-metrics, which constitute a broader class of Finsler metrics that includes all $GDW$-metrics. This new class of Finsler metrics is comprised of two distinct subclasses: generalized weakly-Weyl and generalized $\tilde{D}$-metrics. The generalized weakly-Weyl subclass represents another projective invariant class of Finsler metrics, characterized by a weaker form of Weyl curvature. The subsequent sub-section will provide a comprehensive overview of generalized weakly-Weyl metrics. By consulting \cite{W-Weyl}, readers can access comprehensive descriptions and properties, along with several non-trivial examples related to this class of Finsler metrics. \\
Furthermore, the generalized $\tilde{D}$-metrics are defined by a weaker form of Douglas curvature which have been introduced in the following sub-section.
\subsection{Weakly Generalized Douglas-Weyl Finsler metrics}
Within this section, a generalization of  $GDW$-metrics is introduced, referred to as $W-GDW$-metrics. These new metrics are characterized by the following equation
\[
\tilde{D}_j{^i}_{kl|0}+ \lambda F \tilde{D}_j{^i}_{kl}= U_{jkl} y^i.
\]
Here, $\lambda=\lambda(x,y)$ is a scalar function and $U_{jkl}$ represent the coefficients of a certain tensor field. Additionally, the notation $\tilde{D}_j{^i}_{kl}$ is referred to as the stretch Douglas curvature, and it is defined by $\tilde{D}_j{^i}_{kl}={D}_j{^i}_{kl\mid 0}$.
\begin{thm}
The class of $W-GDW$-metrics is closed under projective change.
\end{thm}
\begin{proof}
Assume that Finsler metrics $F$ and $\bar{F}$ be projective related with the geodesic coefficients of $G^{i}$ and $\bar{G}^{i}$, respectively. We have
\begin{equation}\label{Gi1}
\bar{G}^{i}=G^{i}+P y^{i},
\end{equation}
with projective factor $P$. After differentiating with regards to $y^{j}, y^{k}$ and $y^{l}$, consecutively, we will obtain
\begin{equation}\label{Gij}
\bar{G}^{i} \cdot{ }_{j}=G^{i}{ }_{. j}+P \delta^{i}{ }_{j}+P_{. j} y^{i},
\end{equation}
and
\begin{equation}\label{Gijk}
\bar{G}^{i}{ }_{. j . k}=G^{i}{ }_{. j . k}+P_{. k} \delta^{i}{ }_{j}+P_{. j} \delta^{i}{ }_{k}+P_{. j . k} y^{i},
\end{equation}
and
\begin{equation}\label{Gijkl}
\bar{B}_{j}{ }^{i}{ }_{k l}=B_{j}{ }^{i}{ }_{k l}+P_{. k . l} \delta^{i}{ }_{j}+P_{. j . l} \delta^{i}{ }_{k}+P_{. k . j} \delta^{i}{ }_{l}+P_{. j . k . l} y^{i}.
\end{equation}
Douglas curvature is an invariant quantity, then we denote the Douglas tensor of $F$ and $\bar{F}$, by $D_{j}{ }^{i}{ }_{k l}$. After considering \eqref{Gi1}, \eqref{Gij}, and \eqref{Gijk}, one can come to the following
$$
\begin{gathered}
D_{j}{ }^{i}{ }_{k l \| m}=D_{j}{ }^{i}{ }_{k l \mid m}-D_{j}{ }^{i}{ }_{k l . r}\left(P \delta^{r}{ }_{m}+P_{. m} y^{r}\right)-D_{r}{ }^{i}{ }_{k l}\left(P_{. m} \delta^{r}{ }_{j}+P_{. j} \delta^{r}{ }_{m}+P_{. j . m} y^{r}\right) \\
-D_{j}{ }^{i}{ }_{r l}\left(P_{. m} \delta^{r}{ }_{k}+P_{. k} \delta^{r}{ }_{m}+P_{. k . m} y^{r}\right)-D_{j}{ }^{i}{ }_{k r}\left(P_{. m} \delta^{r}{ }_{l}+P_{. l} \delta^{r}{ }_{m}+P_{. l . m} y^{r}\right) \\
+D_{j}{ }^{r}{ }_{k l}\left(P_{. r} \delta^{i}{ }_{m}+P_{. m} \delta^{i}{ }_{r}+P_{. r . m} y^{i}\right),
\end{gathered}
$$
where \enquote{$\mid\mid$} denotes the horizontal derivative with respect to $\bar{G}$ of $\bar{F}$. Then
\begin{gather}\label{D||andD|}
D_{j}{ }^{i}{ }_{k l|| m}=D_{j}{ }^{i}{ }_{k l \mid m}-P D_{j}{ }^{i}{ }_{k l . m}-P_{. m} D_{j}{ }^{i}{ }_{k l}-P_{. j} D_{m}{ }^{i}{ }_{k l}  \\
-P_{. k} D_{j}{ }^{i}{ }_{m l}-P_{. l} D_{j}{ }^{i}{ }_{k m}+P_{. r} D_{j}{ }^{r}{ }_{k l} \delta^{i}{ }_{m}+P_{. m . r} D_{j}{ }^{r}{ }_{k l} y^{i}.
\end{gather}
By contracting the equation above with $y^{m}$ one obtains
\begin{equation}\label{D||0andD|0}
D_{j}{ }^{i}{ }_{k l|| 0}-D_{j}{ }^{i}{ }_{k l \mid 0}=P_{. r} D_{j}{ }^{r}{ }_{k l} y^{i}.
\end{equation}
With the same as above procedure of calculations we find
$$
\begin{gathered}
D_{j}{ }^{i}{ }_{k l|0| \mid 0}=D_{j}{ }^{i}{ }_{k l|0| 0}-2 P y^{r} D_{j}{ }^{i}{ }_{k l \mid 0 . r}-\left(P \delta^{r}{ }_{j}+P_{. j} y^{r}\right) D_{r}{ }^{i}{ }_{k l \mid 0}\\
-\left(P \delta^{r}{ }_{k}+P_{. k} y^{r}\right) D_{j}{ }^{i}{ }_{r l \mid 0} -\left(P \delta^{r}{ }_{l}+P_{. l} y^{r}\right) D_{j}{ }^{i}{ }_{k r \mid 0}+\left(P \delta^{i}{ }_{r}+P_{. r} y^{i}\right) D_{j}{ }^{r}{ }_{k l \mid 0 .}
\end{gathered}
$$
By simplifying we have
\begin{equation}\label{D|0|0}
D_{j}{ }^{i}{ }_{k l|0| \mid 0}=D_{j}{ }^{i}{ }_{k l|0| 0}-2P D_{j}{ }^{i}{ }_{k l|0}+P_{. r} D_{j}{ }^{r}{ }_{k l \mid 0} y^{i}.
\end{equation}
On the other hands, by taking horizontal derivative \eqref{D||0andD|0} with respect to $\bar{G}$, we derive
\begin{equation}\label{D||0||0}
D_{j}{ }^{i}_{k l\|0\| 0}=D_{j}{ }^{i}_{k l|0| \mid 0}+\left(P_{. r} D_{j}{ }^{r}{ }_{k l}\right)_{\| 0} y^{i}.
\end{equation}
By using the equation \eqref{D|0|0} in the equation \eqref{D||0||0}, one obtains
$$
D_{j}{ }^{i}{ }_{k l|| 0|| 0}=D_{j}{ }^{i}{ }_{k l|0| 0}-2P D_{j}{ }^{i}{ }_{k l|0} +\left(P_{. r} D_{j}{ }^{r}{ }_{k l \mid 0}+\left(P_{. r} D_{j}{ }^{r}{ }_{k l}\right)_{|| 0}\right) y^{i} \text {. }
$$
According to the assumption, $F$ is of $W-GDW$-metric, then there is a function $\lambda$ on $T M$ and some tensor coefficients $U_{j k l}$, such that
$$
D_{j}{ }^{i}{ }_{k l|| 0|| 0}=-(\lambda F+2P) D_{j}{ }^{i}{ }_{k l \mid 0}+U_{j k l} y^{i}+\left(P_{. r} D_{j}{ }^{r}{ }_{k l \mid 0}+\left(P_{. r} D_{j}{ }^{r}{ }_{k l}\right)_{\| 0}\right) y^{i},
$$
which by \eqref{D||0andD|0}, one finds
$$
D_{j}{ }^{i}{ }_{k l \|0 \|0}+\overline{\lambda} F D_{j}{ }^{i}{ }_{k l \| 0}=\bar{U}_{j k l} y^{i}
$$
where $\bar{U}_{j k l}=U_{j k l}- P_{. r} D_{j}{ }^{r}{ }_{k l}+P_{. r} D_{j}{ }^{r}{ }_{k l \mid 0}+\left(P_{. r} D_{j}{ }^{r}{ }_{k l}\right)_{|| 0}$ and $\overline{\lambda}F=\lambda F+2P$. Putting $\tilde{\bar{D}}_{j}{ }^{i}{ }_{k l}=D_{j}{ }^{i}{ }_{k l|| 0}$, we get
\begin{equation}\label{barW-GDW}
\tilde{\bar{D}}_{j}{ }^{i}_{k l \| 0} + \lambda F \tilde{\bar{D}}_{j}{ }^{i}_{k l}=\bar{U}_{j k l} y^{i}.
\end{equation}
\end{proof}
To understand this new class of Finsler metrics, namely the $W-GDW$-metrics, it is essential to explore their connections with other classes of Finsler metrics. The class of $GDW$-metrics serves as an excellent candidate for this examination, due to it is a projective invariant. Next, we introduce the subsequent Proposition for the evaluation.
\begin{prop}
Every $GDW$-metric is a $W-GDW$-metric.
\end{prop}
\begin{proof}
Assume that $F$ is a $GDW$-metric. Then we have
$$
D_{j}{ }^{i}{ }_{k l \mid 0}=T_{j k l} y^{i},
$$
for some tensor's coefficients denoted by $T_{jkl}$. By taking the horizontal derivative with respect to Berwald connection, we obtain $D_{j}{ }^{i}{ }_{k l|0| 0}=T_{j k l \mid 0} y^{i}$. Then
$$
D_{j}{ }^{i}{ }_{k l|0| 0}+\lambda F D_{j}{ }^{i}{ }_{k l \mid 0}=\left(T_{j k l \mid 0}+\lambda F T_{j k l}\right) y^{i},
$$
For some functions $\lambda=\lambda(x, y)$ on $T M$. It means that $F$ is a $W-GDW$-metric.
\end{proof}
The above proposition raises the question of whether there exists a non-trivial $W-GDW$ metric meaning a $W-GDW$-metric that is not a $GDW$-metric. Exploring this question is crucial for a deeper understanding of the class of $W-GDW$-metrics and their unique properties. To investigate this further, we can analyze the defining characteristics of both classes and identify any potential distinctions. By examining specific examples, we may uncover metrics that fall under the $W-GDW$ classification but do not meet the criteria for $GDW$-metrics. In the following theorems, we examine non-trivial $W-GDW$-metrics. While we examine the potential for such metrics in the following theorems, providing explicit examples requires a more comprehensive analysis. This intricate investigation is ongoing and will be presented in a separate paper \cite{ExampleW-Weyl}.
To begin, we will establish a lemma.
\begin{lem}
Let $F=\alpha \varphi(s)$, with $s=\frac{\beta}{\alpha}$, be a regular $(\alpha, \beta)$-metric of non-Randers type, where non-closed 1 -form $\beta$ satisfies the conditions
\begin{equation}\label{rij and sj}
r_{i j}=0, \quad s_{i}=0.
\end{equation}
Then the following equations hold
\be\label{alphaQjkbjbk}
(\alpha Q)_{. j . k \mid 0} b^{j} b^{k}=0.
\ee
\be\label{alphaQjklbjbkbl}
(\alpha Q)_{. j \cdot k \cdot l \mid 0} b^{j} b^{k} b^{l}=0.
\ee
Here, $Q=\frac{\varphi^{\prime}}{\varphi-s \varphi^{\prime}}$ as stated in \eqref{GeoQ}.
\end{lem}

\begin{proof}
To initiate the proof, we first calculate $(\alpha Q)_{. j . k}$ and $(\alpha Q)_{\text {.j.k.l }}$, by using the following equations
\be\label{s.j.k}
\alpha s_{.j.l}= -s \alpha_{.j.l} - \alpha_{.j} s_{.l} - \alpha_{.l} s_{.j}.
\ee
and
\[
(\alpha Q)_{. j}= \alpha_{.j} Q+ Q^{\prime} (b_j-s \alpha_{.j})= (Q- s Q^{\prime}) \alpha_{.j} + b_j Q^{\prime}.
\]
Now, we compute the expressions for $(\alpha Q)_{. j . k}$ and $(\alpha Q)_{. j . k . l}$, which are given by
\begin{equation}\label{alphaQjk}
(\alpha Q)_{. j . k}=\left(Q-s Q^{\prime}\right) \alpha_{. j . k}+\alpha Q^{\prime \prime} s_{. j} s_{. k} .
\end{equation}
and
\be\label{alphaQjkl}
\begin{aligned}
&(\alpha Q)_{. j . k . l}=\left(Q-s Q^{\prime}\right) \alpha_{. j . k . l}-Q^{\prime \prime}\left(s\left[s_{. j} \alpha_{. k . l}+s_{. k} \alpha_{. j . l}+s_{. l} \alpha_{. k . j}\right]\right)\\
&\left.+\left[\alpha_{. j} s_{. k} s_{. l}+\alpha_{. k} s_{. j} s_{. l}+\alpha_{. l} s_{. k} s_{. j}\right]\right)+\alpha Q^{\prime \prime \prime} s_{. j} s_{. k} s_{. l}.
\end{aligned}
\ee
On the other hands, we observe that
$$
F_{\mid 0}=\alpha_{\mid 0} \varphi+\alpha \varphi^{\prime} s_{\mid 0}=0.
$$
By referring to \eqref{GeoCoealphaBeta} and \eqref{rij and sj}, we can find
\be\label{GcoefofLem}
G^i=\bar{G}^i + \alpha Q s{^i}_0.
\ee
Then
\begin{equation}\label{s|0=0}
s_{\mid 0}=s_{\| 0}-2 \alpha Q s^{r}{ }_{0} s_{. r}=s_{\| 0}-2 Q s^{r}{ }_{0}\left(b_{r}-s \alpha_{r}\right)=s_{\| 0}=\frac{1}{\alpha} b_{j \| k} y^{j} y^{k}=\frac{r_{00}}{\alpha}=0,
\end{equation}
where \enquote{$\mid\mid$} denotes the horizontal derivative with respect to $\alpha$. The two equations derived above lead us to conclude that
\begin{equation}\label{alpha|0}
\alpha_{\mid 0}=0.
\end{equation}
Furthermore, utilizing \eqref{s|0=0} and the conditions $F_{. k \mid 0}=0$ and $b_{k\mid 0}= b_{k\mid \mid 0}= s_{k0}$ (note that $r_{ij}=0$), we arrive at the following equation
\begin{equation}\label{Fk|0}
0=F_{. k \mid 0}=\left(\varphi^{\prime} b_{k}+\left(\varphi-s \varphi^{\prime}\right) \alpha_{. k}\right)_{\mid 0}=\varphi^{\prime} s_{k 0}+\left(\varphi-s \varphi^{\prime}\right) \alpha_{. k \mid 0}.
\end{equation}
By contracting the equation above with $b^{k}$ and taking into account \eqref{rij and sj}, we obtain
\begin{equation}\label{alphak|0bk}
\alpha_{. k \mid 0} b^{k}=0.
\end{equation}
It is important to note that, applying \eqref{GcoefofLem}, $b_{k \mid 0}$ can be expressed as
\[
b_{k \mid 0}=b_{k|| 0}-b_{r}\left((\alpha Q)_{. k} s^{r}{ }_{0}+\alpha Q s^{r}{ }_{k}\right)= b_{k|| 0}=s_{k 0}.
\]
Using the above equation, along with \eqref{rij and sj} and \eqref{s|0=0}, we derive
\begin{equation}\label{sj|0bj}
\alpha s_{.j \mid 0} b^{j}=\left(b_{j \mid 0}-s \alpha_{. j \mid 0}\right) b^{j}=0.
\end{equation}
Applying a similar procedure to the equation $F_{. j . k \mid 0}=0$, we have
\[
F_{.j.k\mid 0} b^j b^k = [(\varphi-s \varphi^{\prime}) \alpha_{.j.k\mid 0} + 2 \alpha \varphi^{\prime \prime} s_{.j\mid 0} s_{.k} + (\alpha\varphi^{\prime \prime\prime} s_{.j} s_{.k} -s \varphi^{\prime \prime}\alpha_{.j.k} ) s_{\mid 0}] b^j b^k =0.
\]
We can utilize the previous equations, \eqref{sj|0bj} and \eqref{s|0=0} to arrive at
\begin{equation}\label{alphajk|objbk}
\alpha_{. j . k \mid 0} b^{j} b^{k}=-\alpha \frac{\varphi^{\prime \prime}}{\varphi-s \varphi^{\prime}}\left(s_{. j} s_{. k}\right)_{\mid 0} b^{j} b^{k}=0.
\end{equation}
The equations derived above allow us to express
$$
\begin{gathered}
\alpha \alpha_{. j . k . l \mid 0} b^{j} b^{k} b^{l}=\left(\left(\alpha \alpha_{. j . k}\right)_{. l}-\alpha_{. l} \alpha_{. j . k}\right)_{\mid 0} b^{j} b^{k} b^{l}=\left[\left(\left(\delta_{j k}-\alpha_{. j} \alpha_{. k}\right)_{. l}\right)_{\mid 0}-\left(\alpha_{. l} \alpha_{. j . k}\right)_{\mid 0}\right] b^{j} b^{k} b^{l} \\
=-\left(\alpha_{. k} \alpha_{. j . l}+\alpha_{. j} \alpha_{. l . k}+\alpha_{. l} \alpha_{. j . k}\right)_{\mid 0} b^{j} b^{k} b^{l}=0.
\end{gathered}
$$
This simplifies to
\begin{equation}\label{alphajkl|0bjbkbl}
\alpha_{. j . k . l \mid 0} b^{j} b^{k} b^{l}=0.
\end{equation}
Next, by substituting \eqref{s|0=0} into \eqref{alphaQjk}, we obtain
$$
(\alpha Q)_{. j . k \mid 0}=\left(Q-s Q^{\prime}\right) \alpha_{. j . k \mid 0}+\alpha Q^{\prime \prime}\left(s_{. j} s_{. k}\right)_{\mid 0}.
$$
Utilizing \eqref{sj|0bj} and \eqref{alphajk|objbk} in the equation above leads us to find \eqref{alphaQjklbjbkbl}.\\
Applying a similar procedure to (29), we arrive at
$$
\begin{gathered}
(\alpha Q)_{. j . k . l \mid 0} b^{j} b^{k} b^{l}=\left(Q-s Q^{\prime}\right) \alpha_{. j . k . l \mid 0} b^{j} b^{k} b^{l}-3 Q^{\prime \prime}\left(s\left(s_{. j} \alpha_{. k . l}\right)_{\mid 0}+\left(\alpha_{. j} s_{. k} s_{. l}\right)_{\mid 0}\right) b^{j} b^{k} b^{l} \\
+\alpha Q^{\prime \prime \prime}\left(s_{. j} s_{. k} s_{. l}\right)_{\mid 0} b^{j} b^{k} b^{l}
\end{gathered}
$$
and by incorporating (36), (34), and (35) into the equation, we ultimately arrive at (27).
\end{proof}
Leveraging the lemma established above, we are now prepared to prove the following theorem.
\begin{thm}
Let $F$ be a regular $(\alpha, \beta)$-metric of non-Randers type, where non-closed 1-form $\beta$ satisfies the conditions
\[
r_{i j}=0, \quad s_{i}=0.
\]
Then $F$ is a non-trivial $W-GDW$ metric if and only if
\begin{equation}\label{W-GDWalphaBetaEq}
b^{l}\left(s^{m}{ }_{l|0| 0}+\lambda F s^{m}{ }_{l \mid 0}\right) y_{m}=0\text {, }
\end{equation}
for some scalar function $\lambda=\lambda(x, y)$ defined on $T M$.
\end{thm}
\begin{proof}
Consider the regular $(\alpha, \beta)$-metric given by
$$
F=\alpha \varphi\left(\frac{\beta}{\alpha}\right),
$$
which is of non-Randers type (i.e., $\varphi \neq c_{1} \sqrt{1+c_{2} s^{2}}+c_{3} s$ for any constants $c_{1}>0, c_{2}$, and $c_{3}$ ). Under the conditions $r_{i j}=0$ and $s_{i}=0$, this metric exhibits vanishing $S$-curvature \cite{IsoS-cur}. Furthermore, since $\beta$ is not closed, we can refer to \eqref{GeoCoealphaBeta} to express
$$
G^{i}=\bar{G}^{i}+\alpha Q s_{0}^{i}.
$$
Consequently, we have the expression for the Berwald curvature
\begin{equation}\label{Berwald}
B_{j}{ }^{i}{ }_{k l}=(\alpha Q)_{. j . k} s^{i}{ }_{l}+(\alpha Q)_{. j . l} s^{i}{ }_{k}+(\alpha Q)_{. l . k} s^{i}{ }_{j}+(\alpha Q)_{. j . k . l} s^{i}{ }_{0},
\end{equation}
which indicates that $F$ is not a Berwald metric.\\
On the other hand, according to the main theorem in \cite{GDWalphaBeta}, which states that a Finsler metric $F$ is a $GDW$-metric with vanishing $S$-curvature if and only if it is a Berwald metric, we conclude that the $(\alpha, \beta)$-metric $F$ is not a $GDW$-metric. This conclusion arises from the fact that $F$ is not a Berwald metric while has vanishing $S$-curvature.\\
Now, we will consider the conditions under which this metric may be classified as a $W-G D W$ metric.\\
$F$ has vanishing $S$-curvature then noting to \eqref{D2}, $B_{j}{ }^{i}{ }_{k l}=D_{j}{ }^{i}{ }_{k l}$. Then according to \eqref{Berwald}, $F$ is $W-GDW$-metric if and only if the following equation satisfies for some tensor's coefficients denoted by $U_{jkl}$ and scalar function $\lambda=\lambda(x, y)$ on $T M$,
\[
D_j{^i}_{kl\mid 0\mid 0} + \lambda F D_j{^i}_{kl\mid 0} = B_j{^i}_{kl\mid 0\mid 0} + \lambda F B_j{^i}_{kl\mid 0}=U_{jkl} y^i,
\]
where
\be\label{Ujklyi}
\begin{aligned}
& U_{j k l} y^{i}=A_{j k 0} s^{i}{ }_{l}+A_{l k 0} s^{i}{ }_{j}+A_{j l 0} s^{i}{ }_{k}+A_{j k l 0} s^{i}{ }_{0}+A_{j k} s^{i}{ }_{l \mid 0}+A_{l k} s^{i}{ }_{j \mid 0}+A_{j l} s^{i}{ }_{k \mid 0}+A_{j k l} s^{i}{ }_{0 \mid 0} \\
& (\alpha Q)_{. j . k} s^{i}{ }_{l|0| 0}+(\alpha Q)_{. l . k} s^{i}{ }_{j|0| 0}+(\alpha Q)_{. j . l} s^{i}{ }_{k|0| 0}+(\alpha Q)_{j . k . l} s^{i}{ }_{0|0| 0},
\end{aligned}
\ee
where
$$
\begin{gathered}
A_{j k}=2(\alpha Q)_{. j . k \mid 0}+\lambda F(\alpha Q)_{. j . k} \\
A_{j k l}=2(\alpha Q)_{. j \cdot k . l \mid 0}+\lambda F(\alpha Q)_{. j \cdot k . l} \\
A_{j k 0}=(\alpha Q)_{. j . k|0| 0}+\lambda F(\alpha Q)_{. j . k \mid 0} \\
A_{j k l 0}=(\alpha Q)_{. j . k . l|0| 0}+\lambda F(\alpha Q)_{. j . k . l \mid 0}
\end{gathered}
$$
By contracting the equations presented above with respect to $b^{j}, b^{k}$, and $b^{l}$, and utilizing the previous lemma, we arrive at the following results
\begin{equation}\label{Ajkbjbk}
A_{j k} b^{j} b^{k}=\lambda F(\alpha Q)_{. j \cdot k} b^{j} b^{k}, \quad A_{j k l} b^{j} b^{k} b^{l}=\lambda F(\alpha Q)_{. j . k \cdot l} b^{j} b^{k} b^{l}.
\end{equation}
Next, by applying these results and considering the assumption that $s_{j}=0$ in equation \eqref{Ujklyi}, we can conclude that
\be\label{Ujklbjbkblyi}
\begin{aligned}
&U_{j k l} b^{j} b^{k} b^{l} y^{i}=A_{j k l 0} b^{j} b^{k} b^{l} s^{i}{ }_{0}+3(\alpha Q)_{. j . k} b^{j} b^{k}\left(s^{i}{ }_{l|0| 0}+2\lambda F s^{i}{ }_{l \mid 0}\right) b^{l}\\
&+(\alpha Q)_{j . k . l} b^{j} b^{k} b^{l}\left(s^{i}{ }_{0|0| 0}+2\lambda F s^{i}{ }_{0 \mid 0}\right).
\end{aligned}
\ee
Contracting the above equation with $y_{i}=\alpha \alpha_{. i}$ yields
$$
\alpha^{2} U_{j k l} b^{j} b^{k} b^{l}=3(\alpha Q)_{. j . k} b^{j} b^{k} y_{i}\left(s^{i}{ }_{l|0| 0}+2\lambda F s^{i}{ }_{l \mid 0}\right) b^{l} .
$$
Substituting this expression back into the earlier equation, we obtain
\[
-3(\alpha Q)_{. j . k} b^{j} b^{k} b^{l}\left(s^{m}{ }_{l|0| 0}+2\lambda F s^{m}{ }_{l \mid 0}\right)^{\alpha} h^{i}{ }_{m}=A_{j k l 0} b^{j} b^{k} b^{l} s^{i}{ }_{0}
\]
\[
+(\alpha Q)_{j . k . l} b^{j} b^{k} b^{l}\left(s^{i}{ }_{0|0| 0}+2\lambda F s^{i}{ }_{0 \mid 0}\right),
\]
where ${ }^{\alpha} h^{i}{ }_{m}=\delta^{i}{ }_{m}-\frac{1}{\alpha}\alpha_{. m} y^{i}$. By contracting the above equation again with ${ }^{F} y_{i}=F F_{. i}$, and using the relation $(\alpha Q)_{. j \cdot k} b^{j} b^{k}=\frac{b^{2}-s^{2}}{\alpha}\left(Q-s Q^{\prime}+\left(b^{2}-s^{2}\right) Q^{\prime \prime}\right)$ and
\[
s{^i}_0 F F_{.i} = \varphi \varphi^{\prime} b_i s{^i}_0 + \varphi(\varphi-s \varphi^{\prime}) \alpha_{.i} s{^i}_0 =0,
\]
we arrive at
\[
\left(Q-s Q^{\prime}+\left(b^{2}-s^{2}\right) Q^{\prime \prime}\right)^{F} y_{m}\left(s^{m}{ }_{l|0| 0}+2\lambda F s^{m}{ }_{l \mid 0}\right) b^{l}=0.
\]
Solving this equation leads us to two possibilities: either $Q-s Q^{\prime}+\left(b^{2}-s^{2}\right) Q^{\prime \prime}=0$ or ${ }^{F} y_{m}\left(s{^m}_{l|0| 0}+2\lambda F s{^m}_{l \mid 0}\right) b^{l}=0$. We will demonstrate that the first possibility is untenable. If we assume $Q-s Q^{\prime}+\left(b^{2}-s^{2}\right) Q^{\prime \prime}=0$, then we can rearrange this to obtain
$$
Q^{\prime \prime}-\frac{s}{b^{2}-s^{2}} Q^{\prime}+\frac{1}{b^{2}-s^{2}} Q=0,
$$
which implies
$$
Q=k_{1} s+k_{2} \sqrt{b^{2}-s^{2}},
$$
for real constant $k_{1}$ and $k_{2}$. Given that $Q=\frac{\varphi^{\prime}}{\varphi-s \varphi^{\prime}}$, the equation simplifies to
$$
\left(1+k_{1} s^{2}+k_{2} s \sqrt{b^{2}-s^{2}}\right) \varphi^{\prime}=\left(k_{1} s+k_{2} \sqrt{b^{2}-s^{2}}\right) \varphi.
$$
Upon solving this equation, we find
$$
\varphi=c\operatorname{exp}\left[\int_{0}^{s} \frac{k_{1} t+k_{2} \sqrt{b^{2}-t^{2}}}{1+t\left(k_{1} t+k_{2} \sqrt{b^{2}-t^{2}}\right)} d t\right],
$$

where $k_{1}$ and $k_{2}$ are real constants, and $c>0$ is a non-zero constant. This indicates that $F$ is a singular Finsler metric, which contradicts our initial assumption. Consequently, we have
$$
\begin{gathered}
0=F_{.m}\left(s^{m}{ }_{l|0| 0}+2\lambda F s^{m}{ }_{l \mid 0}\right) b^{l}=\left(\varphi^{\prime} b_{m}+\left(\varphi-s \varphi^{\prime}\right) \frac{y_{m}}{\alpha}\right)\left(s{^m}{ }_{l|0| 0}+\lambda F s{^m}_{l \mid 0}\right) b^{l} \\
=\left(\varphi-s \varphi^{\prime}\right) \frac{y_{m}}{\alpha}\left(s^{m}{ }_{l|0| 0}+\lambda F s^{m}{ }_{l \mid 0}\right) b^{l}
\end{gathered}
$$
Since $F$ is not of Randers type, we easily conclude \eqref{W-GDWalphaBetaEq}.
\end{proof}
To further explore the question mentioned above, we will consider the class of non-trivial $D$-recurrent Finsler metrics \cite{D-rec}. Our aim is to establish the necessary and sufficient conditions for these metrics to qualify as non-trivial $W-G D W$-metrics.
\begin{thm}
Let $F$ be a $D$-recurrent Finsler metric on an $n$-dimensional manifold $M$, such that
$$
D_{j}{ }^{i}{ }_{k l \mid 0}=\sigma D_{j}{ }^{i}{ }_{k l},
$$
where $\sigma$ is a non-zero 1-homogenous smooth scalar function on $T M_{0}$ . $F$ is a non-trivial $W-GDW$-metric if and only if
\begin{equation}\label{WGDWDrecEq}
\sigma_{\mid 0}+\sigma^{2}+\lambda F \sigma=0,
\end{equation}
for some scalar function $\lambda=\lambda(x, y)$.
\end{thm}
\begin{proof}
Let $F$ be a Finsler metric that is classified as a $D$-recurrent metric, satisfying
\begin{equation}\label{D-rec}
D_{j}{ }^{i}{ }_{k l \mid 0}=\sigma D_{j}{ }^{i}{ }_{k l}.
\end{equation}
From this, we obtain the following equation
$$
D_{j}{ }^{i}{ }_{k l|0| 0}=\left(\sigma_{\mid 0}+\sigma^{2}\right) D_{j}{ }^{i}{ }_{k l}.
$$
According to the definition of $W-GDW$-metrics, $F$ qualifies as a $W-GDW$-metric if and only if there exist scalar function $\lambda=\lambda(x, y)$ and the certain tensor's coefficients $U_{jkl}$ such that
\begin{equation}\label{D-recUjklyi}
U_{j k l} y^{i}=D_{j}{ }^{i}{ }_{k l|0| 0}+\lambda F D_{j}{ }^{i}{ }_{k l \mid 0}=\left(\sigma_{\mid 0}+\sigma^{2}+\lambda F \sigma\right) D_{j}{ }^{i}{ }_{k l}.
\end{equation}
Put
$$
D_{j}{ }^{i}{ }_{k l \mid 0}=T_{j k l} y^{i}+d_{j}{^i}_{kl},
$$
where $d_{j}{ }^{i}{ }_{k l} \neq 0$ and does not contain any coefficients of $y^{i}$, indicating that $F$ is not a $G D W$ metric. Based on \eqref{D-rec} and noting that $\sigma$ is a non-zero scalar function, we have
$$
D_{j}{ }^{i}{ }_{k l}=\frac{1}{\sigma}\left(T_{j k l} y^{i}+d_{j}{ }^{i}{ }_{k l}\right).
$$
Substituting the aforemention equation in \eqref{D-recUjklyi} gives
$$
U_{j k l} y^{i}=\frac{1}{\sigma}\left(\sigma_{\mid 0}+\sigma^{2}+\lambda F \sigma\right)\left(T_{j k l} y^{i}+d_{j}{^i}_{k l}\right).
$$
From this equation, we concludes the equation \eqref{WGDWDrecEq}.
\end{proof}
The significant class of $G D W$-metrics encompasses two important subclasses of Finsler metrics: Douglas and Weyl metrics. As mentioned earlier, we have generalized the class of $G D W$ metrics to include $W-G D W$-metrics. To relate this to the earlier classes of Finsler metrics, we introduce two new subclasses that generalize Douglas and Weyl metrics, respectively. Notably, both of these subclasses are subsets of $W-G D W$-metrics. In the following sections, we will explore these two intriguing classes of Finsler metrics.
\subsubsection{\textbf{Generalized $\tilde{D}$-metrics}}
Berwald developed the concept of stretch curvature as a generalization of Landsberg curvature \cite{Ber1} as
\[
\Sigma_{ijkl}=2(L_{ijk|l}-L_{ijl|k}).
\]
Additionally, Berwald formulated various Finsler metric classes, including Landsberg metrics and stretch metrics in 1928 \cite{Ber2}. Deriving from the Berwald curvature through covariant horizontal differentiation along Finslerian geodesics, the particular quantity
\[
\widetilde{B}_y: T_pM \times T_pM \times T_pM \longrightarrow T_pM, \quad y \in T_pM, \quad p \in M,
\]
\[
\widetilde{B}_y (u, v, w)= \widetilde{B}_j{^i}_{kl}u^i v^j w^k \frac{\partial}{\partial x^i}|_p, \quad u, v, w \in T_pM,
\]
where $\widetilde{B}_j{^i}_{kl}=B_j{^i}_{kl|0}$, has been introduced in \cite{Sh2}. In \cite{KozmaNew}, a presentation was made on two classes of stretch Finsler metrics that are defined with respect to $\widetilde{B}$ (and $\widetilde{H}_{jk}=H_{jk|m}y^m$). Within this section, our objective is to present a stretch tensor known as stretch Douglas which introduced in the previous section. Furthermore, we will introduce a fresh category of Finsler metric that exhibits the same level of projective invariance. \\
We define a metric based on the expanded notion of Douglas curvature, termed $\tilde{D}$-metric, has been introduced in Preliminary Section \cite{Dtilde-Sadegh}. Based on this notation, we introduce generalized $\tilde{D}$-metrics as follow. \\
A Finsler metric is called generalized $\tilde{D}$-metric if
\be\label{GDtilde}
\tilde{D}_j{^i}_{kl|0}+ \mu_r \tilde{D}_j{^r}_{kl} y^i=0,
\ee
for some tensor's coefficients denoted by $\mu_r$ and $\tilde{D}_j{^i}_{kl}=D_j{^i}_{kl|0}$.

From the definition of $W-GDW$ metrics, setting $U_{jkl}=-\mu_r \tilde{D}j{^i}{kl}$ and $\lambda=0$ directly leads to the following proposition.

\begin{prop}\label{GDtildinWGDW}
Every generalized $\tilde{D}$-metric is a $W-GDW$-metric.
\end{prop}
The subsequent Lemma assists in establishing a link between this novel category of Finsler metrics and other significant categories, providing additional instances for examination.
\begin{lem}\label{WeylDThetaLemma}\cite{W-Weyl}
For every Finsler metric F, with Weyl curvature $W=\left\{W_{y}\right\}_{y \in T_{x} M \backslash 0}$ and Douglas curvature $D=\left\{D_{y}\right\}_{y \in T_{x} M \backslash 0}$, it follows that
\begin{equation}\label{WtildandDandTheta}
W_{j}{ }^{i}{ }_{m l . k} y^{m}=D_{j}{ }^{i}{ }_{k l \mid 0}-\frac{1}{n+1} \theta_{j k l} y^{i},
\end{equation}
where $\theta_{j k l}=2 E_{j k \mid l}-\frac{1}{3}\left(R^{s}{ }_{l . s}-(n+2) R_{. l}\right)_{. j . k}$ and $W_{j}{ }^{i}{ }_{kl}=\frac{1}{3}(W{^i}_{k.l}-W{^i}_{l.k})_{.j}$.
\end{lem}
Examples presented in the following demonstrate the existence of non-trivial generalized $\tilde{D}$-metric forms.
\begin{ex}\cite{Randers} Put
$$
\Omega=\left\{(x, y, z) \in R^{3} \mid x^{2}+y^{2}+z^{2}<1\right\}, \quad p=(x, y, z) \in \Omega, \quad y=(u, v, w) \in T_{p} \Omega,
$$
Define the Randers metric $F=\alpha+\beta$ by
$$
\alpha=\frac{\sqrt{(-y u+x v)^{2}+\left(u^{2}+v^{2}+w^{2}\right)\left(1-x^{2}-y^{2}\right)}}{1-x^{2}-y^{2}}, \quad \beta=\frac{-y u+x v}{1-x^{2}-y^{2}}.
$$
The above Randers metric has vanishing flag curvature $K=0$ and $S$-curvature $S=0$. Then $R{^i}_k=0$ and $E_{jk}=0$, it means that $\theta_{jkl}=0$. $F$ has zero Weyl curvature and $\beta$ is not closed then $F$ is not of Douglas type. According to \eqref{WtildandDandTheta} in Lemma \ref{WeylDThetaLemma}.
$$
\tilde{D}_{j}{ }^{i}{ }_{k l}=\frac{1}{n+1} \theta_{j k l} y^{i}=0.
$$
It means that the equation \eqref{GDtilde} holds and it is of generalized $\tilde{D}$-metric.
\end{ex}
\begin{cor}\label{DinGDtild}
The class of Douglas metrics is a proper subset of the class of generalized $\tilde{D}$-metrics.
\end{cor}

In the paper \cite{Dtilde-Sadegh}, the class of relatively isotropic $\tilde{D}$-metrics is introduced as Finsler metrics that satisfy the following equation,
$$
\tilde{D}_{j}{ }^{i}_{kl\mid 0}+\lambda F \tilde{D}_{j}{ }^{i}{ }_{k l}=0,
$$
where $\lambda=\lambda(x, y)$ is scalar function on $T M$. In the following lemma, it is shown that
\begin{lem}\label{GDtildandRID}
Let $F$ be a GDW-metric. It is relatively isotropic $\tilde{D}$-metric if and only if it is a generalized $\tilde{D}$-metric.
\end{lem}
\begin{proof}
Let us assume that $F$ is a $GDW$-metric. Under this assumption, we can express the following relationship
\[
D_j{^i}_{kl|0}= T_{jkl} y^i,
\]
where $T_{jkl}$ represents certain coefficients of a tensor. Next, for a scalar function defined as $\lambda = \lambda(x,y)$, we have the equation
\be\label{RIDEq}
\tilde{D}_{j}{ }^{i}_{kl\mid 0}+\lambda F \tilde{D}_{j}{ }^{i}{ }_{k l}= (T_{jkl|0}+\lambda F T_{jkl})y^i,
\ee
Additionally, for some tensor coefficients denoted by $\mu_r$, we can formulate
\be\label{GDtildEq}
\tilde{D}_{j}{ }^{i}_{k l \mid 0}+\mu_r \tilde{D}_{j}{ }^{r}{ }_{k l}y^i= (T_{jkl|0}+\mu_0 T_{jkl})y^i.
\ee
Now, if we consider $F$ to be a relatively isotropic $\tilde{D}$-metric, we can derive from equation \eqref{RIDEq} that
\[
T_{jkl|0} = -\lambda F T_{jkl}.
\]
Substituting this expression into equation \eqref{GDtildEq} yields
\[
\tilde{D}_{j}{^i}_{k l \mid 0} + \mu_r \tilde{D}_{j}{^r}_{k l} y^i = (\mu_0 - \lambda F) T_{jkl} y^i.
\]
By selecting $\mu_r = \lambda F_{.r}$, we can conclude that $F$ is a generalized $\tilde{D}$-metric. The converse statement holds true as well.
\end{proof}
According to above Lemma, we discover that the example presented in the paper \cite{Dtilde-Sadegh}, which is a Weyl metric but not of relatively isotropic Douglas type, yields a Weyl metric which is not a generalized $\tilde{D}$-metric.
\begin{ex} (\cite{Dtilde-Sadegh}, \cite{Sh3})
Let us consider the Randers metric $F=\alpha+\beta$ which is given by
$$
\alpha=\frac{\sqrt{\left(1-|a|^{2}|x|^{2}\right)|y|^{2}+\left(|x|^{2}<a, y>-2<a, x><x, y>\right)^{2}}}{1-|a|^{2}|x|^{2}}
$$
and
$$
\beta=-\frac{|x|^{2}<a, y>-2<a, x><x, y>}{1-|a|^{2}|x|^{2}}
$$
$F$ is of isotropic $S$-curvature, $S=(n+1) c F$, with $c=<a, x>$, and of scalar flag curvature $\lambda=\lambda(x,y)$ as
$$
\lambda=3 \frac{c_{; 0}}{F}+3 c^{2}-2|a|^{2}|x|^{2}
$$
However, we have
\be\label{ajkandbk}
\begin{array}{r}
a_{j k}=\frac{\delta_{j k}}{\Delta}+b_{j} b_{k}, \\
b_{k}=2 \frac{c}{\Delta} x_{k}-\frac{|x|^{2}}{\Delta} c_{; k}.
\end{array}
\ee
where $\Delta=1-|a|^{2}|x|^{2}$. Using Maple for the computation, which has been done in \cite{Sh3}, we have
\be\label{sjkandsj}
\begin{array}{r}
s_{j k}=\frac{2}{\Delta^{2}}\left(c_{; k} x_{j}-c_{; j} x_{k}\right), \\
s_{k}=2 \frac{|a|^{2}|x|^{2}}{\Delta} x_{k}+2 \frac{c}{\Delta} c_{; k}.
\end{array}
\ee
and
\begin{equation}\label{Gialpha}
G^{i}={ }^{\alpha} G^{i}+P y^{i}+\alpha s^{i}{ }_{0},
\end{equation}
where $P=c(\alpha-\beta)-s_{0}$.
\end{ex}
Now, let us assume $n=3$, constant vector $a=(-1,0,0), X=(x, y, z)$ and $Y=(u, v, w)$. Then we have $c=-x$ and $c_{; k}=-\delta_{1 k}$. Hence, we have
\[
s_{j k}=-\frac{2}{\Delta^{2}}\left(\delta_{1 k} x_{j}-\delta_{1 j} x_{k}\right),
\]
which indicates that $\beta$ is not closed. Consequently, $F$ is not a Douglas metric, even though it has scalar flag curvature. Therefore, it is a Weyl metric and, as a result, a $GDW$-metric.\\
Based on the analysis conducted in \cite{Dtilde-Sadegh}, we conclude that $F$ is not a relatively isotropic $\tilde{D}$ metric. Consequently, by the aforementioned lemma, it follows that $F$ is also not a generalized $\tilde{D}$-metric

This finding then allows us to derive the following corollary.
\begin{cor}\label{WeylnotGDtild}
There is a Weyl metric which is not of generalized $\tilde{D}$-metrics.
\end{cor}
\subsubsection{\textbf{Generalized Weakly-Weyl Finsler metrics}}
Within this particular area, a fresh category of projective invariant Finsler metrics called generalized weakly-Weyl Finsler metrics is being introduced, all created utilizing the Weyl metric class. Although not strictly classified as a subset of the $GDW$-metrics, this new class does have intersection with the class of $GDW$ metrics. Here, we explore some properties and characteristics of generalized weakly-Weyl Finsler metrics. More properties of these metrics, along with their relationships to other established classes of Finsler metrics, have been considered in \cite{W-Weyl}.

A Finsler metric $F$ is called generalized weakly-Weyl Finsler metric if it satisfies the following equation.
\begin{equation*}
\tilde{W}_j{^i}_{kl|0}+\lambda F \tilde{W}_j{^i}_{kl} = \mu_r \tilde{W}_j{^r}_{kl} y^i,
\end{equation*}
for some tensors $\mu_r$ and smooth scalar function $\lambda$ on $TM$. Note that, $\tilde{W}_{j}{ }^{i}{ }_{k l}={W}_{j}{ }^{i}{ }_{ml.k}y^m$ and ${W}_{j}{ }^{i}{ }_{k l}=\frac{1}{3}\left(W^{i}{ }_{k . l}-W^{i}{ }_{l . k}\right)_{. j}$ is referred to as the weakly-Weyl curvature, which has been extensively introduced in [14]. First of all, we have
\begin{thm}\cite{W-Weyl}
The class of generalized weakly-Weyl Finsler metrics is closed under projective changes.
\end{thm}
The class of generalized weakly-Weyl Finsler metrics is a projective invariant in Finsler geometry, similar to the class of Weyl metrics. Although Weyl metrics are a subset of this novel class, there exist generalized weakly-Weyl metrics that are not Weyl metrics. In \cite{Example}, the authors introduce several spherically symmetric Finsler metrics that, while not classified as Weyl metrics, are nonetheless weakly-Weyl metrics and also Douglas metrics.
\begin{ex} \cite{Example}
Spherically symmetric Finsler metrics in $\mathbb{B}^{n}(v) \subseteq \mathbb{R}^{n}$, has been introduced as $F=|y| \varphi\left(|x|, \frac{<x, y>}{|y|}\right)$ with $\varphi:[0, v) \times \mathbb{R} \rightarrow \mathbb{R}$ where $(x, y) \in T \mathbb{B}^{n}(v) \backslash\{0\}$ and
$$
\varphi(r, s)=s . h(r)-\frac{s}{\left(a+b r^{2}\right)^{\lambda}} \int_{s_{0}}^{s} \sigma^{-2} f\left(\frac{r^{2}-\sigma^{2}}{\left(a+b r^{2}\right)^{\lambda}}\right) d \sigma,
$$
where $a$, $b$ and $\lambda$ are constants satisfying $a+b r^{2}>0$. These metrics are not Weyl metrics but are weakly-Weyl and also Douglas metrics. Their Weyl curvature is as follows
$$
\begin{gathered}
W{^i}_{k }=\frac{4 \lambda(\lambda-1) b^{2}}{\left(a+b r^{2}\right)^{2}}\left(x_{j} x^{i} \delta_{k l}+\frac{1}{m-1} x_{k} x_{l} \delta^{i}{ }_{j}+\frac{|x|^{2}}{m-1} \delta^{i}{ }_{k} \delta_{j l}-\delta_{j l} x_{k} x^{i}\right. \\
\left.-\frac{|x|^{2}}{m-1} \delta^{i}{ }_{j} \delta_{k l}-\frac{1}{m-1} x_{j} x_{k} \delta^{i}{ }_{l}\right) y^{k} y^{l}=\omega_j{^i}_{kl}(x) y^j y^l.
\end{gathered}
$$
\end{ex}
Then one could concludes
\begin{cor}\label{WeylinGWWeyl}
The class of Weyl Finsler metrics is a proper subset of the class of generalized weakly-Weyl Finsler metrics.
\end{cor}
According to Lemma \ref{WeylDThetaLemma} and definition of generalized weakly-Weyl Finsler metrics, one finds
\begin{thm}\label{GWWeylinWGDW}
Every generalized weakly-Weyl Finsler metric is a $W-GDW$-metric.
\end{thm}
\begin{proof}
Assume that $F$ be a generalized weakly-Weyl Finsler metric. Then
$$
\tilde{W}_{j}{ }^{i}{ }_{k l \mid 0}+ \lambda F \tilde{W}_j{^i}_{kl} = \mu_{r} \tilde{W}_{j}{ }^{r}{ }_{k l} y^{i}.
$$
Based on the Lemma \ref{WeylDThetaLemma}, we have
\[
\tilde{W}_{j}{ }^{i}_{kl|0} + \lambda F \tilde{W}_{j}{ }^{i}_{kl}- \mu_{r} \tilde{W}_{j}^{r}{ }_{k l} y^{i}=D_{j}{ }^{i}{ }_{k l\mid0\mid 0}+ \lambda F D_{j}{ }^{i}{ }_{k l\mid0}
\]
\[
- \mu_{r} D_{j}^{r}{ }_{k l \mid 0} y^{i} -\frac{1}{n+1}\left(\theta_{j k l \mid 0}+(\lambda F-\mu_{0}) \theta_{j k l}\right) y^{i}=0.
\]
Hence, based on the equations above, we obtain
$$
D_{j}{ }^{i}{ }_{k l|0| 0}+\lambda F D_{j}{ }^{i}{ }_{k l \mid 0} =U_{j k l} y^{i}
$$
where $U_{j k l}=\frac{1}{n+1}\left(\theta_{j k l \mid 0}+(\lambda F-\mu_{0}) \theta_{j k l}\right)+ \mu_{r} D_{j}{ }^{r}{ }_{k l \mid 0}$. It means that $F$ is a $W-G D W$-metric.
\end{proof}


\begin{thebibliography}{}

\bibitem{KozmaNew}
S. A. Abbas, L. Kozma, {\it On New Classes of Stretch Finsler Metrics}, Journal of Finsler Geometry and its Applications, {\bf 3}, 1 (2022), 86-99.

\bibitem{D-rec}
M. Atashafrouz, B. Najafi, {\it On D-Recurrent Finsler Metrics},  Bulletin of the Iranian Mathematical Society, {\bf 47} (2021), 143–156.

\bibitem{Proj. 1}
S. B\'{a}cs\'{o}, M. Matsumoto, {\it Projective change between Finsler spaces with $(\alpha, \beta)$-metric}, Tensor N.S., {\bf 55} (1994), 252--257.

\bibitem{Bacso}
S. B\'{a}cs\'{o}, M. Matsumoto, {\it On Finsler spaces of Douglas type, A generalization of notion of Berwald space}, Publicationes Mathematicae Debrecen, {\bf 51} (1997), 385-406.

\bibitem{Ber1}
L. Berwald, {\it Uber Parallel \"{u}bertragung in R\"{a}umen mit allgemeiner Ma$\beta$bestimmung}, Jahresbericht der Deutschen Mathematiker-Vereinigung, {\bf 34} (1925), 2013-2020.

\bibitem{Ber2}
L. Berwald, {\it Parallel \"{u}bertragung in allgemeinen R\"{a}umen}, Atti del Congresso Internazionale dei Matematici, Bologna, {\bf 4} (1928), 263-270.

\bibitem{DShen}
X. Cheng, Z. Shen, {\it On Douglas metrics}, Publicationes Mathematicae Debrecen, {\bf 66} 3-4 (2005), 503-512.

\bibitem{Randers}
X. Cheng, Z. Shen, {\it Randers metrics with special curvature properties}, Osaka Journal of Mathematics, {\bf 40} (2003), 87-101.

\bibitem{IsoS-cur}
X. Cheng, Z. Shen, {\it A class of Finsler metrics with isotropic $S$-curvature}, Israel Journal of Mathematics, {\bf 169} (2009), 317-340.

\bibitem{Douglas}
J. Douglas, {\it The general geometry of paths}, Annals of Mathematics, {\bf 29} (1927), 143-168.

\bibitem{Example}
H. Liu, X. Mo, {\it Examples of Finsler metrics with special curvature properties}, Mathematische Nachrichten, {\bf 13} (2015), 1527-1537.

\bibitem{GDW}
B. Najafi, Z. Shen and A. Tayebi, {\it On a projective class of Finsler metrics}, Publicationes Mathematicae Debrecen, {\bf 70} (2007), 211--219.

\bibitem{Proj. 2}
B. Najafi, A. Tayebi, {\it A new quantity in Finsler geometry}, Comptes Rendus de l'Académie des Sciences, Series I. {\bf 349} (2011), 81--83.

\bibitem{Dtilde-Sadegh}
N. Sadeghzadeh,  {\it A New Class of Finsler Metrics: Douglas Curvature and Its Generalizations}, FILOMAT, accepted.

\bibitem{W-Weyl}
N. Sadeghzadeh, M. Yavari, {\it Generalized weakly-Weyl Finsler metrics: A generalized approach to Sakaguchi's Theorem}, Submitted.

\bibitem{ExampleW-Weyl}
N. Sadeghzadeh, M. Yaghoubi, M. Yavari, {\it Weakly-Weyl Metrics in the Context of Finslerian Curvature}, Submitted.

\bibitem{Proj. 3}
Z. Shen, {\it Projectively flat Finsler metrics of constant flag curvature}, Transactions of the American Mathematical Society, {\bf 355} (2003), 1713-1728.

\bibitem{Sh2}
Z. Shen, {\it Differential Geometry of Spray and Finsler Spaces}, Kluwer Academic Publishers, 2001.

\bibitem{Sh3}
Z. Shen, {\it On some non-Riemannian quantities in Finsler geometry}, Canadian Mathematical Bulletin, {\bf 56} (2013), 184–193.

\bibitem{Proj. 4}
Y. Shen, Y. Yu, {\it On projectively related Randers metrics}, International Journal of Mathematics, {\bf 19} (2008), 503--520.

\bibitem{Shiing}
Shiing-Shen Chern, {\it Finsler geometry is just Riemannian geometry without the quadratic restriction}, Notices of AMS, 1996.

\bibitem{GDWalphaBeta}
A. Tayebi, H. Sadeghi, {\it On Generalized Douglas-Weyl $(\alpha,\beta)$-metrics}, Acta Mathematica Sinica, {\bf 31}, 10 (2015), 1611-1620.

\bibitem{Proj. 5}
A. Tayebi, M. Shahbazi Nia, {\it A new class of projectively flat Finsler metrics with constant flag curvature $K = 1$}, Differential Geometry and its Applications, {\bf 41} (2015), 123--133.


\end{thebibliography}
\end{document}